\documentclass{amsart}

\usepackage{amssymb,amsmath,amsthm,amsfonts,mathtools}
\usepackage{setspace}
\usepackage{graphicx}
\usepackage[usenames,dvipsnames,svgnames,table]{xcolor}
\usepackage[margin=1.5in]{geometry}
\usepackage{enumerate}
\usepackage{comment}
\usepackage{cases}
\usepackage{enumitem} 
\usepackage{latexsym}
\usepackage{hyperref}
\hypersetup{
  colorlinks   = true,    
  urlcolor     = blue,    
  linkcolor    = blue,    
  citecolor    = purple      
}

\newtheorem{thm}{Theorem}[section]

\newtheorem{lem}[thm]{Lemma}

\newcommand{\N}{\mathbb{N}}
\newcommand{\R}{\mathbb{R}}
\newcommand{\C}{\mathbb{C}}

\newcommand{\Z}{\mathbb{Z}}

\newcommand{\lb}{\langle}
\newcommand{\rb}{\rangle}

\renewcommand{\t}{\widetilde}

\renewcommand{\hat}{\widehat}

\allowdisplaybreaks[4]

\title[Global well-posedness for the Schr\"odinger-Kawahara system]{Global well-posedness for the coupled system of Schr\"odinger and Kawahara equations}

\author{Wangseok Shin}
\address{Department of Mathematics, University of Illinois, Urbana, IL 61801, U.S.A.}
\email{wshin14@illinois.edu}
\date{}

\begin{document}
\maketitle

\begin{abstract} 
We study the local and global well-posedness for the coupled system of Schr\"odinger and Kawahara equations on the real line. The Sobolev space $L^{2} \times H^{-2}$ is the space where the lowest regularity local solutions are obtained. The energy space is $H^1 \times H^2$. We apply the Colliander-Holmer-Tzirakis method \cite{CHT} to prove the global well-posedness in $L^2 \times L^2$ where the energy is not finite. Our method generalizes the method of Colliander-Holmer-Tzirakis in the sense that the operator that decouples the system is nonlinear.
\end{abstract}

\section{Introduction}

In this paper, we consider the Cauchy problem for the coupled Schr\"odinger-Kawahara system on the real line $\R$,
\begin{align}\label{eq: SK}
\begin{cases}
i\partial_{t} u + \partial^{2}_{x}u = \alpha uv +\beta |u|^2u,  \\
\partial_{t} v + \gamma \partial^{3}_{x}v-\delta \partial^{5}_{x}v + v\partial_{x}v = \varepsilon\partial_{x}|u|^2, \\
u(x,0) = u_0(x), \quad v(x,0) = v_0(x),
\end{cases} & (x,t) \in \R^+ \times \R^+,
\end{align}
where $u=u(x,t) \in \C$ ,$v=v(x,t) \in \R$, and $\alpha, \beta, \gamma, \delta, \varepsilon$ are real constants. The system (\ref{eq: SK}) is a special case of the more general system 
\begin{equation}\label{eq: General}
\begin{cases}
i\partial_{t} S + ic_{1}\partial_{x}S + \partial^{2}_{x}S = c_{2} SL +c_{3} |u|^2u,  \\
\partial_{t} L + c_{4}\partial_{x} L+ P(\partial_{x})L + c_{5} L\partial_{x}L = c_{6} |S|^2,
\end{cases}
\end{equation}
where $c_{1}, c_{2}, c_{3}, c_{4}, c_{5}, c_{6}$ are real constants and $P$ is a constant coefficient polynomial. In a physical context, the system (\ref{eq: General}) describes various phenomena related to the interaction between short and long waves; here $S=S(x,t) \in \C$ and $L=L(x,t) \in \R$ models the short and long waves respectively.  See \cite{BOP1, Makhankov}, and the references therein for more details. 

When $\delta=0$, the system (\ref{eq: SK}) reduces to the more famous Schr\"odinger-Korteweg-de Vries (SKdV) system. Local and global well-posedness for the SKdV system has been extensively examined in several studies, for example \cite{BOP, BOP1, CL, GW, Pecher, Wu}. The endpoint local theory was obtained in \cite{GW}, where the authors proved the local well-posedness for $(u_0,v_0) \in L^{2} \times H^{-\frac34}$ using the Besov-type Fourier restriction spaces with some low frequency modifications. More recently, global well-posedness for the SKdV system with positive coupling interactions ($\alpha \varepsilon >0$, $\delta=0$ in (\ref{eq: SK})) was obtained for $(u_0,v_0) \in H^{\frac12+} \times H^{\frac12+}$ in \cite{Wu} utilizing the $I$-method. Finally we note that a system which couples the fourth-order Schr\"odinger equation and the fifth-order KdV equation was studied in \cite{ACF}. 

The aim of this paper is to prove local and global well-posedness results for the system (\ref{eq: SK}) in the case $\delta \neq 0$. By local well-posedness, we mean existence, uniqueness, and continuous dependence of the solutions with respect to the initial data locally in time. We say that the system (\ref{eq: SK}) is globally well-posed if it is locally well-posed and the local solutions can be extended to arbitrarily large time. Our main results are the following.

\begin{thm} \label{thm: LWP}
The system (\ref{eq: SK}) is locally well-posed for $(u_0,v_0) \in H^{s}(\R) \times H^{k}(\R)$ for any $(s,k)$ such that $s \geq 0$, $k\geq-2$, and $s-4 \leq k<\min(8s+1,s+2)$, provided $\delta \neq 0$.
\end{thm}
\begin{thm} \label{thm: GWP}
The system (\ref{eq: SK}) is globally well-posed for $(u_0,v_0) \in L^{2}(\R) \times L^2(\R)$, provided $\delta \neq 0$.
\end{thm}

To prove Theorem \ref{thm: LWP}, we combine the arguments given in \cite{GW, Kato, Wu}. The low-regularity local well-posedness result given here is sharp in the sense that the solution maps of the cubic nonlinear Schr\"odinger equation and the Kawahara equation fail to be uniformly continuous in $H^{0-}$ and $H^{-2-}$ respectively, see \cite{Kato, KPV}. 

The proof of Theorem \ref{thm: GWP} is based on the argument given in \cite{CHT}. In \cite{CHT} the authors provide a globalization scheme for the abstract evolution system
\begin{equation*}
\begin{cases}
Ku = F(u,v),  \\
Lv = G(u), \\
u(x,0) \in L^{2}, \quad v(x,0) \in H^{s},
\end{cases} 
\end{equation*}
where $K, L$ are linear differential operators of evolution type and $F,G$ are nonlinearities. Here the linear semigroups corresponding to $K, L$ are assumed to be unitary, and $u$ conserves the $L^{2}$-norm. In the present paper we generalize this scheme to prove the global well-posedness for the system (\ref{eq: SK}), where $L$ is the nonlinear Kawahara evolution $Lv=\partial_{t} v + \gamma \partial^{3}_{x}v-\delta \partial^{5}_{x}v + v\partial_{x}v$. This paper is the first to apply this technique when the operator $L$ is nonlinear.

There are at least three conserved quantities for the system (\ref{eq: SK}):
\begin{align}\label{eq: mass conservation}
    \mathcal{M}(t):=\int_{\R} |u|^{2} dx,
\end{align}
\begin{align}\label{eq: momentum conservation}
    \mathcal{Q}(t):=\int_{\R} \alpha v^{2}+2\varepsilon \operatorname{Im}(u \partial_{x}\overline{u}) dx,
\end{align}
and
\begin{align}\label{eq: energy conservation}
    \mathcal{E}(t):=\int_{\R} \alpha \varepsilon |u|^{2}v-\frac{\alpha}{6}v^{3}+\frac{\beta \varepsilon}{2}|u|^4+\frac{\alpha \gamma}{2}|\partial_{x}v|^{2}-\frac{\alpha \delta}{2}|\partial^{2}_{x}v|^{2}+\varepsilon |\partial_{x}u|^{2} dx.
\end{align}
In our globalization argument we do not take advantage of (\ref{eq: momentum conservation}) and (\ref{eq: energy conservation}). We only make use of (\ref{eq: mass conservation}) and the $L^2$-norm conservation of the Kawahara equation demonstrated in Lemma \ref{lem: Kawahara conservation}. This allows us to prove the global well-posedness for the system (\ref{eq: SK}) in $L^2 \times L^2$ for arbitrary coupling coefficients $\alpha, \beta, \gamma, \delta, \varepsilon$, regardless of their signs provided $\delta \neq 0$.

Finally we note that our globalization argument does not directly apply to the $\delta=0$ case. The main difficulty in this case is the failure of the estimate
\begin{align*}
    \|\partial_{x}(v_1 v_2)\|_{Y^{0,-c,\infty}} \lesssim \|v_1\|_{Y^{0,b,1}} \|v_2\|_{Y^{0,b,1}}, \quad 0<b,c<\frac12, \quad 2b+c=1
\end{align*}
when $\delta=0$, which is essential for our argument (see Section \ref{sec: notations} for the definitions of the functions spaces). We therefore leave this case open.

This paper is organized as follows. In Section \ref{sec: notations} we define some function spaces and fix the notation we will use throughout this paper. In Section \ref{sec: prelim} we recall some basic estimates and provide proofs of some technical inequalities. In section \ref{sec: LWP} we prove the local well-posedness for the system (\ref{eq: SK}). Finally in Section \ref{sec: GWP} we establish the global theory.  

\section{Notations \& Function spaces} \label{sec: notations}

Define the one-dimensionl Fourier transform as
\begin{align*}
    \hat{f}(\xi)=\int_{\R}e^{-ix\xi}f(x)dx.
\end{align*} 
If $f(x,t)$ is a function on $\R \times \R$ we define its space-time Fourier transform $\hat{f}(\xi,\tau)$ by
\begin{align*}
    \hat{f}(\xi,\tau)=\int_{\R \times \R}e^{-i(x\xi+t\tau)}u(x,t)dx dt.
\end{align*} 
The inverse Fourier transforms are given by
\begin{align*}
    & g^{\vee}(x):=\frac{1}{2\pi}\int_{\R}e^{i x \xi}g(\xi)d\xi, \\
    & g^{\vee}(x,t):=\frac{1}{(2\pi)^2}\int_{\R \times \R}e^{i (x \xi+ t \tau)}g(\xi,\tau)d\xi d\tau.
\end{align*}
Let $\lb \xi \rb :=\sqrt{1+|\xi|^{2}}$. Define the Sobolev space $H^{s}(\R)$ by the norm
\begin{align*}
 \| f \|_{H^{s}(\R)}=\left (\int_{\R}\lb \xi \rb^{2s}|\hat{f}(\xi)|^{2} d\xi \right )^{\frac{1}{2}}.
\end{align*}
Let $p(\xi)=\gamma \xi^3+\delta \xi^5$. Let $S(t)=e^{it\partial_{x}^{2}}$ and $W(t)=e^{it p(\partial_{x}/i)}$. We define the Fourier restriction norm spaces $X^{s,b}$ and $Y^{s,b}$ by the norms
\begin{align*}
    &\| u \|_{X^{s,b}} = \left\|\lb \xi \rb^{s} \lb \tau + \xi^2 \rb^b \hat{u}(\xi, \tau) \right\|_{L^2_{\xi,\tau}},\\
    &\| u \|_{Y^{s,b}} = \left\|\lb \xi \rb^{s} \lb \tau - p(\xi) \rb^b \hat{u}(\xi, \tau) \right\|_{L^2_{\xi,\tau}}.
\end{align*}
Let $\varphi_{0}:\R \to [0,1]$ denote an even smooth function such that $\varphi_0(\xi)=1$ for $|\xi| \leq \frac54$ and $\varphi_0(\xi)=0$ for $|\xi| > \frac32$. For $j \in \N$, we set $\varphi_{j}(\xi)=\varphi_{0}(\xi/2^{j})-\varphi_{0}(\xi/2^{j-1})$.
We will need the following Besov-type modifications of Fourier restriction norm spaces introduced in \cite{CKS} and applied in our context in \cite{BHHT, GP}:
\begin{align*}  
    & \|u\|_{X^{s,b,1}}=\sum_{j=0}^{\infty}2^{bj}\|\lb \xi \rb^{s}\varphi_{j}(\tau+\xi^2)\hat{u}(\xi, \tau)\|_{L^2_{\xi,\tau}}, \\
    & \|u\|_{Y^{s,b,1}}=\sum_{j=0}^{\infty}2^{bj}\|\lb \xi \rb^{s}\varphi_{j}(\tau-p(\xi))\hat{u}(\xi, \tau)\|_{L^2_{\xi,\tau}},
\end{align*}
and
\begin{align*} 
    & \|u\|_{X^{s,b,\infty}}=\sup_{j \geq 0}2^{bj}\|\lb \xi \rb^{s}\varphi_{j}(\tau+\xi^2)\hat{u}(\xi, \tau)\|_{L^2_{\xi,\tau}}, \\
    & \|u\|_{Y^{s,b,\infty}}=\sup_{j \geq 0}2^{bj}\|\lb \xi \rb^{s}\varphi_{j}(\tau-p(\xi))\hat{u}(\xi, \tau)\|_{L^2_{\xi,\tau}}.
\end{align*}
Let $\psi$ be a smooth function such that $\psi=1$ on $[-1,1]$ and $\psi = 0$ outside $[-2,2]$. We denote $\psi_{T}(t):=\psi(t/T)$. For $N \geq 0$, $P^N f$ will denote the high frequency projection $(\chi_{|\xi|\geq N}\hat{f}(\xi))^{\vee}$. Finally, in this paper we abbreviate the H\"older inequality
\begin{align*}
    \left |\int f_1\cdots f_n d\xi d\tau \right | \leq \|f_1\|_{L^{p_1}_{\xi,\tau}}\cdots \|f_1\|_{L^{p_1}_{\xi,\tau}}, \quad \frac{1}{p_1}+\cdots+\frac{1}{p_n}=1
\end{align*}
by the $L^{p_1}\dots L^{p_n}$-H\"older inequality.

\section{Preliminary estimates} \label{sec: prelim}

The following estimates for the Schr\"odinger evolution are well known.
\begin{lem} \cite[Lemma 2.1]{Takaoka} \label{lem: Schrodinger group estimates}
We have
 \begin{align*}
     & \| S(t)g \|_{L^{q}_{t}L^{r}_{x}} \lesssim \|g \|_{L^{2}}, \quad \frac{2}{q}+\frac{1}{r}=\frac{1}{2}, \quad 2\leq q,r \leq \infty \qquad \text{(Strichartz estimate)}, \\
     & \| S(t)g \|_{L^{4}_{x}L^{\infty}_{t}} \lesssim \|g \|_{H^{\frac14}} \qquad \text{(Maximal function estimate)}.
 \end{align*}
\end{lem}

For the Kawahara evolution, we have the following estimates.

\begin{lem} \cite[Lemma 2.3]{Huo} \label{lem: Kawahara group estimates}
Let $p(\xi)=\gamma \xi^3+\delta \xi^5$ and assume that $\delta \neq 0$. Then there exists $N>0$ only depending on $\delta, \varepsilon$ such that we have
 \begin{align*}
     & \left \| P^{N}W(t)g \right \|_{L^{6}_{t,x}} \lesssim \|g \|_{H^{-\frac12}} \qquad \text{(Strichartz estimate)}, \\
     &  \left \| P^{N}W(t)g \right \|_{L^{\infty}_{x}L^{2}_{t}} \lesssim \|g \|_{H^{-2}} \qquad \text{(Kato smoothing estimate)}, \\
     & \left \| P^{N}W(t)g \right \|_{L^{4}_{x}L^{\infty}_{t}} \lesssim \|g \|_{H^{\frac14}} \qquad \text{(Maximal function estimate)}. \\ 
\end{align*}
\end{lem}

The following lemma allows us to transfer the above estimates to the Fourier restriction norm spaces.

\begin{lem} \cite[Proposition 1.2]{GP} \label{lem: extension}
    Let $Z \subset \mathcal{S}'(\R \times \R)$ be a Banach space of space-time functions such that
    \begin{align*}
        \|g(t)u(x,t)\|_{Z} \leq \|g\|_{L^{\infty}}\|u\|_{Z}
    \end{align*}
    holds for any $g \in L^{\infty}(\R)$ and $u \in Y$. Assume that the estimate
    \begin{align*}
        \left \|P^N S(t)g \right \|_{Z} \lesssim \|g\|_{H^{s}} 
    \end{align*}
    holds true for some $N \geq 0$ and any $g \in H^{s}$. Then we have
    \begin{align*}
        \left \|P^N u \right \|_{Z} \lesssim \|u\|_{X^{s,\frac12,1}}
    \end{align*}
    for any $u \in X^{s,\frac12,1}$.  A similar statement holds if $S(t)$ and $X^{s,\frac12,1}$ are respectively replaced by $W(t)$ and $Y^{s,\frac12,1}$.
\end{lem}
In \cite{GP} Lemma \ref{lem: extension} is proved only for $N=0$, however the same proof readily applies to the general case. Also note that this lemma directly implies the embeddings $X^{s,\frac12,1} \hookrightarrow C_{t}H^s_{x}$ and $Y^{s,\frac12,1} \hookrightarrow C_{t}H^s_{x}$. From Lemmas \ref{lem: Schrodinger group estimates}, \ref{lem: Kawahara group estimates}, and \ref{lem: extension}, we obtain the following estimates:
\begin{lem}
We have
 \begin{align}\label{eq: Schrodinger strichartz}
        \| f \|_{L^{q}_{t}L^{r}_{x}} \lesssim \|f \|_{X^{0,\frac12,1}}, \quad \frac{2}{q}+\frac{1}{r}=\frac{1}{2}, \quad 2\leq q,r \leq \infty,
\end{align}
\begin{align} \label{eq: Sch maximal}
    \| f \|_{L^{4}_{x}L^{\infty}_{t}} \lesssim \|f\|_{X^{\frac14,\frac12,1}}.
\end{align}
Also, for $N>0$ given in Lemma \ref{lem: Kawahara group estimates} we have
\begin{align}\label{eq: Kawahara L6 strichartz}
         \left \| P^N f \right \|_{L^{6}_{t,x}} \lesssim \|f \|_{Y^{-\frac12,\frac12,1}},
    \end{align}
\begin{align} \label{eq: Kawahara Kato}
    \left \| P^N f \right \|_{L^{\infty}_{x}L^{2}_{t}} \lesssim \|f\|_{Y^{-2,\frac12,1}},
\end{align}
\begin{align} \label{eq: Kawahara maximal}
    \left \| P^N f \right \|_{L^{4}_{x}L^{\infty}_{t}} \lesssim \|f\|_{Y^{\frac14,\frac12,1}}.
\end{align}
\end{lem}
Next, we present some useful estimates obtained from the complex interpolation of the above estimates. For justification of such interpolations, see \cite{GP}. Note that since $\ell^{1}(\Z) \hookrightarrow \ell^{2}(\Z)$, we have
\begin{align} \label{eq: Plancherel}
    \|f\|_{L^{2}_{x,t}} \lesssim \|f\|_{X^{0,0,1}} \quad \text{and} \quad \|f\|_{L^{2}_{x,t}} \lesssim \|f\|_{Y^{0,0,1}}.
\end{align}
Interpolating (\ref{eq: Schrodinger strichartz}) for $(q,r)=(6,6)$ with (\ref{eq: Plancherel}), we obtain
\begin{align}  \label{eq: Schrodinger L3}
        \|f\|_{L^{3}_{t,x}} \lesssim \|f \|_{X^{0,\frac14,1}},
\end{align}
\begin{align}  \label{eq: Schrodinger L4}
        \|f\|_{L^{4}_{t,x}} \lesssim \|f \|_{X^{0,\frac38,1}}.
\end{align}
Similarly, interpolation between (\ref{eq: Kawahara L6 strichartz}) and (\ref{eq: Plancherel}) gives
\begin{align} \label{eq: Kawahara L3}
        \left \| P^N f \right \|_{L^{3}_{t,x}} \lesssim \|f \|_{Y^{-\frac14,\frac14,1}},
\end{align}
\begin{align} \label{eq: Kawahara L4}
        \left \| P^N f \right \|_{L^{4}_{t,x}} \lesssim \|f \|_{Y^{-\frac38,\frac38,1}}.
\end{align}
Interpolating (\ref{eq: Schrodinger strichartz}) for $(q,r)=(12,3)$ with (\ref{eq: Plancherel}), we have
\begin{align}  \label{eq: Schrodinger interpolation for cubic 1}
        \|f\|_{L^{\frac{24}{5}}_{t}L^{4}_{x}} \lesssim \|f \|_{X^{0,\frac13,1}}.
\end{align}
Interpolating (\ref{eq: Schrodinger strichartz}) for $(q,r)=(4,\infty)$ with (\ref{eq: Plancherel}), we have
\begin{align}  \label{eq: Schrodinger interpolation for cubic 2}
        \|f\|_{L^{\frac{8}{3}}_{t}L^{4}_{x}} \lesssim \|f \|_{X^{0,\frac14,1}}.
\end{align}

The proof of the following lemma is much similar to that of \cite[Lemma 2.2]{Wu}.
\begin{lem} \label{lemma: Wu}
    We have
    \begin{align} \label{eq: multiplier 1}
    \left | \iint_{\begin{subarray}{l}\\ \xi=\xi_{1}+\xi_{2}\\ \tau=\tau_{1}+\tau_{2} \end 
{subarray}} | 5\delta \xi^{4}+3\gamma \xi^{2} +2 \xi_1 |^{\frac{1}{2}} \hat{f}(\xi, \tau) 
\hat{g}(\xi_1, \tau_1)\hat{h}(-\xi_{2}, -\tau_{2})  \right | \lesssim \| f\|_{Y^{0,\frac12,1}} \|g\|_{X^{0,\frac12,1}} \|h\|_{X^{0,0,1}}.
\end{align}
\end{lem}
\begin{proof}
Let 
\begin{align*}
    I(f,g,h)=\iint_{\begin{subarray}{l}\\ \xi=\xi_{1}+\xi_{2}\\ \tau=\tau_{1}+\tau_{2} \end 
{subarray}} | 5\delta \xi^{4}+3\gamma \xi^{2} +2 \xi_1 |^{\frac{1}{2}} \hat{f}(\xi, \tau)
\hat{g}(\xi_1, \tau_1)\hat{h}(-\xi_{2}, -\tau_{2}).
\end{align*}
Let $\hat{f_{j}}=\varphi_{j}(\tau-p(\xi))\hat{f}(\xi, \tau)$ and $\hat{g_{k}}=\varphi_{k}(\tau+\xi^2)\hat{g}(\xi, \tau)$. Setting $\tau=\lambda+p(\xi)$, $\tau_{1}=\lambda_{1}-\xi_{1}^2$, and $\tau_{2}=\lambda_{2}+\xi_{2}^2$, we have 
\begin{multline*}
    I(f_j,g_k,h)=\iint_{\begin{subarray}{l}\\ \xi=\xi_{1}+\xi_{2}\\ \tau=\tau_{1}+\tau_{2} \end 
{subarray}}  |5\delta \xi^{4}+3\gamma \xi^{2} +2 \xi_1|^{\frac{1}{2}} \hat{f_j}(\xi, \lambda+p(\xi))
\hat{g_k}(\xi_1, \lambda_{1}-\xi_{1}^2) \\
\times \hat{h}(\xi_{1}-\xi,\lambda_{1}-\lambda-\xi_{1}^2-p(\xi))d\xi d\xi_1 d\lambda d\lambda_1.
\end{multline*}
We make the change of variables $(\eta,\omega)=T(\xi,\xi_1)$, where $\eta=\xi-\xi_{1}$, $\omega=\lambda-\lambda_{1}+p(\xi)+\xi_{1}^2$. Let $J$ denote the Jacobian of this change of variables. Then
\begin{align*}
    |J|=|5\delta \xi^{4}+3\gamma \xi^{2} +2 \xi_1|.
\end{align*}
Let $H(\eta,\omega,\lambda,\lambda_1)=(\hat{f_j}\hat{g_k}) \circ T^{-1}(\eta,\omega,\lambda,\lambda_1)$. Then
\begin{align*}
        |I(f_j,g_k,h)| 
        & \leq \int \frac{|H(\eta,\omega,\lambda,\lambda_1)|}{|J|^{\frac12}}|\hat{h}(-\eta,-\omega)|d\eta d\omega d\lambda d\lambda_1 \\
        & \leq \int \left ( \int  \frac{|H(\eta,\omega,\lambda,\lambda_1)|^{2}}{|J|}d\eta d\omega \right )^{\frac12}d\lambda d\lambda_1 \|h\|_{L^2} \\
        &= \int \| \varphi_{j}(\lambda)\hat{f}(\xi,\lambda+p(\xi)) \|_{L^{2}_{\xi}}d\lambda \int \| \varphi_{k}(\lambda_1)\hat{g}(\xi_1,\lambda_1-\xi_1) \|_{L^{2}_{\xi_1}}d\lambda_1 \|h\|_{L^2}.
\end{align*}
Since by the Cauchy-Schwartz inequality
\begin{align*} 
    \int_{\R} \| \varphi_{j}(\lambda)\hat{f}(\xi,\lambda+p(\xi)) \|_{L^{2}_{\xi}}d\lambda 
    &= \int_{|\lambda| \sim 2^j} \| \varphi_{j}(\lambda)\hat{f}(\xi,\lambda+p(\xi)) \|_{L^{2}_{\xi}}d\lambda\\
    &\lesssim 2^{\frac{j}{2}}\| \varphi_{j}(\tau-p(\xi))\hat{f}(\xi,\tau) \|_{L^{2}_{\xi,\tau}},
\end{align*}
and similarly 
\begin{align*}
    \int \| \varphi_{k}(\lambda_1)\hat{g}(\xi_1,\lambda_1-\xi_1) \|_{L^{2}_{\xi_1}}d\lambda_1 \lesssim 2^{\frac{k}{2}}\| \varphi_{j}(\tau+\xi^2)\hat{f}(\xi,\tau) \|_{L^{2}_{\xi,\tau}},
\end{align*}
we have
\begin{align*}
    |I(f,g,h)| \leq \sum_{j,k \geq 0} |I(f_j,g_k,h)| \lesssim \|f\|_{Y^{0,\frac12,1}}\|g\|_{X^{0,\frac12,1}}\|h\|_{L^2}.
\end{align*}
\end{proof}

\begin{lem} \label{lem: key estimate}
We have
    \begin{align*}
        \left | \iint_{\begin{subarray}{l}\\ \xi=\xi_{1}+\xi_{2}\\ \tau=\tau_{1}+\tau_{2} \end 
        {subarray}} | 5\delta \xi^{4}+3\gamma \xi^{2}+2 \xi_1 |^{\frac{1}{4}}\hat{f}(\xi, \tau)\hat{g}(\xi_1, \tau_1)\hat{h}(-\xi_2,-\tau_2)  \right | \lesssim \| f\|_{Y^{0,\frac14,1}} \|g\|_{X^{0,\frac38,1}} \|h\|_{X^{0,\frac14,1}}.
    \end{align*}
\end{lem}

\begin{proof}
By (\ref{eq: Schrodinger strichartz}) for $(q,r)=(6,6)$, (\ref{eq: Schrodinger L3}), and the $L^2L^3L^6$-H\"older inequality, we have\begin{align} \label{eq: multiplier 2}
    \left | \iint_{\begin{subarray}{l}\\ \xi=\xi_{1}+\xi_{2}\\ \tau=\tau_{1}+\tau_{2} \end 
{subarray}} \hat{f}(\xi, \tau)
\hat{g}(\xi_1, \tau_1)\hat{h}(-\xi_2,-\tau_2)  \right | \lesssim \| f\|_{Y^{0,0,1}} \|g\|_{X^{0,\frac14,1}} \|h\|_{X^{0,\frac12,1}}.
\end{align}
Interpolating (\ref{eq: multiplier 1}) and (\ref{eq: multiplier 2}), we have
    \begin{multline} \label{eq: Wu's estimate}
        \left | \iint_{\begin{subarray}{l}\\ \xi=\xi_{1}+\xi_{2}\\ \tau=\tau_{1}+\tau_{2} \end 
        {subarray}} |5\delta \xi^{4}+3\gamma \xi^{2}+2 \xi_1 |^{\frac{\theta}{2}}\hat{f}(\xi, \tau)\hat{g}(\xi_1, \tau_1)\hat{h}(-\xi_2,-\tau_2)\right | \\ \lesssim \| f\|_{Y^{0,\frac{\theta}{2},1}} \|g\|_{X^{0,\frac{1-\theta}{4}+\frac{\theta}{2},1}} \|h\|_{X^{0,\frac{1-\theta}{2},1}}, \quad 0<\theta<1,
    \end{multline}
see \cite{GM}. By taking $\theta=\frac12$, we obtain the desired bound.
\end{proof}

\begin{lem} \label{lem: key estimate 2}
Let $N>0$ be as in Lemma \ref{lem: Kawahara group estimates}. We have
\begin{align*} 
 \iint_{\begin{subarray}{l}\\ \xi=\xi_{1}+\xi_{2}\\ \tau=\tau_{1}+\tau_{2} \end 
        {subarray}}  \chi_{\{|\xi|\geq N\}}|\hat{f}(\xi,\tau)||\hat{g}(\xi_1,\tau_1)||\hat{h}(-\xi_{2},-\tau_{2})| \lesssim \| f\|_{Y^{-\frac{7}{16},\frac14,1}} \|g\|_{X^{\frac{1}{16},\frac18,1}} \|h\|_{X^{0,\frac38,1}}.
\end{align*}
\end{lem}
\begin{proof}
Interpolating (\ref{eq: Sch maximal}) with (\ref{eq: Plancherel}), we have
\begin{align} \label{eq: Sch interpolation 4}
     \| f \|_{L^{\frac{16}{7}}_{x}L^{\frac83}_{t}} \lesssim \|f\|_{X^{\frac{1}{16},\frac18,1}}.
\end{align} 
Interpolating (\ref{eq: Kawahara Kato}) with (\ref{eq: Plancherel}), we have
\begin{align} \label{eq: Kawa interpolation 1}
     \left \| P^N f \right \|_{L^{\frac52}_{x}L^{2}_{t}} \lesssim \|f\|_{Y^{-\frac25,\frac{1}{10},1}}.
\end{align}
Interpolating (\ref{eq: Kawahara L6 strichartz}) with (\ref{eq: Kawa interpolation 1}), we obtain
\begin{align} \label{eq: Kawa interpolation 3}
     \left \| P^N f \right \|_{L^{\frac{16}{5}}_{x}L^{\frac{8}{3}}_{t}} \lesssim \|f\|_{X^{-\frac{7}{16},\frac14,1}}.
\end{align}
Using (\ref{eq: Schrodinger L4}), (\ref{eq: Sch interpolation 4}) and (\ref{eq: Kawa interpolation 3}), the desired estimate follows from the H\"older inequality.
\end{proof}
Finally we record a calculus lemma which is often used below.
\begin{lem} \cite[Lemma 3.3]{ET} \label{CalcLem} 
Let 
\begin{align*}
\phi_{\alpha}(a):=
  \begin{cases}
  1 & \alpha>1 \\
  \log (1+\lb a \rb) & \alpha =1 \\
  \lb a \rb^{1-\alpha} & \alpha <1.
  \end{cases}
\end{align*}
For $\alpha \geq \beta \geq 0$ such that $\alpha+\beta>1$, we have
\begin{align*}
  \int\frac{dx}{\lb x-a_1 \rb^{\alpha}\lb x-a_2 \rb^{\beta}} \lesssim \lb a_1-a_2 \rb^{-\beta}\phi_{\alpha}(a_1-a_2),
\end{align*}
\end{lem}

\section{Local theory} \label{sec: LWP}

\subsection{Linear and bilinear estimates}
We follow the analysis in \cite{Kato} for this section. Let $p_{\lambda}(\xi)=\gamma \lambda^2 \xi^3+\delta \xi^5$. We denote by $S_{\lambda}(t)=e^{it \lambda^3 \partial_{x}^{2}}$ and $W_{\lambda}(t)=e^{it p_{\lambda}(\partial_{x}/i)}$. Let
\begin{align*}
    \|u\|_{X^{s,b}_{\lambda}}:=\left\|\lb \xi \rb^{s} \lb \tau +\lambda^3 \xi^2 \rb^b \hat{u}(\xi, \tau) \right\|_{L^2_{\xi,\tau}}.
\end{align*}
For
\begin{align*}
     & P_1:=\{ (\xi,\tau) \in \R^2 : |\tau-p_{\lambda}(\xi)| \leq \frac{31}{32}\delta |\xi|^5+\frac{7}{8} \lambda^2 \gamma |\xi|^3 \}, \\
     & P_2:=\R^2\setminus P_1,
\end{align*}
let
\begin{align*}
    w_{k,\lambda,\theta}(\xi,\tau):= \lb \xi \rb^{k} \lb \tau-p_{\lambda}(\xi)\rb^{\frac12+2\theta} \chi_{P_1}(\xi,\tau)+\lb \xi \rb^{k+1+15 \theta}\lb \tau-p_{\lambda}(\xi)\rb^{\frac{3}{10}-\theta} \chi_{P_2}(\xi,\tau).
\end{align*}
We define the space $Y^{k}_{\lambda,\theta}$ by the norm
\begin{align*}
    \|u\|_{Y^{k}_{\lambda,\theta}}:=\|w_{k,\lambda,\theta}(\xi,\tau) \hat{u}(\xi,\tau) \|_{L^2_{\xi,\tau}}+\|u\|_{L^{\infty}_t H^k_x}.
\end{align*}
Note that since $w_{k,\lambda,\theta}(\xi,\tau) \lesssim \lb \xi \rb^{k} \lb \tau-p_{\lambda}(\xi)\rb^{\frac12+2\theta}$ and $Y^{k}_{\lambda,\theta} \hookrightarrow C_t H^k_x$, we have
\begin{align} \label{eq: Y embedding}
    \|u\|_{Y^{k}_{\lambda,\theta}} \lesssim \|u\|_{Y^{k,\frac12+2\theta}_{\lambda}}.
\end{align}
Let
\begin{align*}
    \| u \|_{Z^{k}_{\lambda,\theta}}:=\left\| \lb \tau-p_{\lambda}(\xi) \rb^{-1} w_{s,\lambda,\theta}(\xi,\tau) \hat{u}(\xi,\tau) \right \|_{L^2_{\xi,\tau}} + \left\| \lb \tau-p_{\lambda}(\xi) \rb^{-1} \right \lb \xi \rb^{s} \hat{u}\|_{L^2_{\xi}L^1_{\tau}}.
\end{align*}
\begin{lem}\cite[Proposition 2.5]{Kato}
    For $0<\lambda \leq 1$, we have
\begin{align*}
    &\left \| \psi(t) S_{\lambda}(t) g\right\|_{X^{s,b}_{\lambda}} \lesssim \|g \|_{H^s}, \\
    &\left \| \psi(t) W_{\lambda}(t) g\right\|_{Y^{k}_{\lambda,\theta}} \lesssim \|g \|_{H^s}.
\end{align*}
The implicit constants do not depend on $\lambda$.
\end{lem}

\begin{lem} \cite[Proposition 2.6]{Kato}
For $0<\lambda \leq 1$, we have
\begin{align*}
   \left \| \psi(t) \int_{0}^{t} W_{\lambda}(t-t')F(t')dt' \right \|_{Y^{k}_{\lambda,\theta}} \lesssim \|F\|_{Z^{k}_{\lambda,\theta}}.
\end{align*}
The implicit constant does not depend on $\lambda$.
\end{lem}

\begin{lem} \label{lem: multilinear lwp}
There exists $0<\theta<1$ and $\kappa >0$ such that the followings hold.
    \begin{enumerate}[label=(\roman*)]
    \item For $0 \leq s\leq k+4$, we have $\| uv \|_{X^{s,-\frac12+2\theta}} \lesssim \|u\|_{X^{s,\frac12+\theta}}\|v\|_{Y^{k,\frac12+\theta}}$.
    \item For $0 \leq s < k+\frac{3}{2}$ and $k \geq -2$, we have $\| uv \|_{X^{s,-\frac12+2\theta}_{\lambda}} \lesssim \lambda^{-\kappa}\|u\|_{X^{s,\frac12+\theta}_{\lambda}}\|v\|_{Y^{k}_{\lambda,\theta}}$.
    \item For $s \geq 0$, we have $\| u_1 \overline{u_2} u_3 \|_{X^{s,-\frac12+2\theta}_{\lambda}} \lesssim  \lambda^{-\kappa}\|u_1\|_{X^{s,\frac12+\theta}_{\lambda}}\|u_2\|_{X^{s,\frac12+\theta}_{\lambda}}\|u_3\|_{X^{s,\frac12+\theta}_{\lambda}}$.
    \item For $k > -\frac32$, we have $\| \partial_{x}(v_1 v_2) \|_{Y^{k,-\frac12+\theta}} \lesssim \|v_1\|_{Y^{k,\frac12+\theta}}\|v_2\|_{Y^{k,\frac12+\theta}}$.
    \item For $k \geq -2$, we have $\| \partial_{x}(v_1 v_2) \|_{Z^{k}_{\lambda,\theta}} \lesssim \|v_1\|_{Y^{k}_{\lambda,\theta}}\|v_2\|_{Y^{k}_{\lambda,\theta}}$.
    \item For $k<\min(8s+1,s+2)$, we have $\| \partial_{x}(u_1 \overline{u_2}) \|_{Y^{k,-\frac12+2\theta}_{\lambda}} \lesssim \lambda^{-\kappa} \|u_1\|_{X^{s,\frac12+\theta}_{\lambda}}\|u_2\|_{X^{s,\frac12+\theta}_{\lambda}}$.
    \end{enumerate}
    The implicit constants do not depend on $\lambda$.
\end{lem}
We only prove $(ii)$ here. Using Lemma \ref{lemma: Wu}, $(i)$ and $(vi)$ can be proved much similar to \cite[Lemma 3.3, Lemma 3.4]{Wu}. The proof of $(iii)$ can be found in \cite{Bourgain}. The estimates $(iv)$ and $(v)$ are proved in \cite{CLMW} and \cite{Kato} respectively. See \cite{GW} for numerically tracking the dependence on $\lambda$. 
\begin{proof}[Proof of (ii)]
It suffices to bound the quantities
\begin{align} \label{eq: uv dual A}
\iint_{\begin{subarray}{l}\\ \xi=\xi_{1}+\xi_{2}\\ \tau=\tau_{1}+\tau_{2} \end
{subarray}} \frac{\lb \xi \rb^{s}|\hat{f}(\xi, \tau)|}{\lb \tau + \lambda^3 \xi^2 \rb^{\frac12+2\theta}}\frac{|\hat{g}(\xi_1,\tau_1)|}{\lb \xi_1 \rb^{s}\lb \tau_1 + \lambda^3\xi_1^2 \rb^{\frac12+\theta}}\frac{\chi_{P_1}(\xi_2,\tau_2)|\hat{h}(\xi_2,\tau_2)|}{\lb \xi_2 \rb^{k}\lb \tau_2 - p_{\lambda}(\xi_2) \rb^{\frac12+\theta}}    
\end{align}
and 
\begin{align} \label{eq: uv dual B}
\iint_{\begin{subarray}{l}\\ \xi=\xi_{1}+\xi_{2}\\ \tau=\tau_{1}+\tau_{2} \end
{subarray}} \frac{\lb \xi \rb^{s}|\hat{f}(\xi, \tau)|}{\lb \tau + \lambda^3 \xi^2 \rb^{\frac12+2\theta}}\frac{|\hat{g}(\xi_1,\tau_1)|}{\lb \xi_1 \rb^{s}\lb \tau_1 + \lambda^3\xi_1^2 \rb^{\frac12+\theta}}\frac{\chi_{P_2}(\xi_2,\tau_2)|\hat{h}(\xi_2,\tau_2)|}{\lb \xi_2 \rb^{k}\lb \tau_2 - p_{\lambda}(\xi_2) \rb^{\frac{3}{10}-\theta}}    
\end{align}
by $\lambda^{-\frac32}\| f\|_{L^2} \|g\|_{L^2} \|h\|_{L^2}$. Since (i) implies the bound of (\ref{eq: uv dual A}), we only consider (\ref{eq: uv dual B}). Note that
\begin{align*}
    \frac{\chi_{P_2}(\xi_2,\tau_2)\lb \xi \rb^{s}}{\lb \xi_1 \rb^{s}\lb \xi_2 \rb^{k}\lb \tau_2 - p_{\lambda}(\xi_2) \rb^{\frac{3}{10}-\theta}} 
    & \lesssim \frac{\lb \xi \rb^{s}}{\lb \xi_1 \rb^{s}\lb \xi_2 \rb^{k}\lb \frac{31}{32}\delta |\xi_2|^5+\frac{7}{8} \lambda^2 \gamma |\xi_2|^3 \rb^{\frac{3}{10}-\theta}} \\
    & \lesssim 
     \begin{cases}
        \lb \xi \rb^{s}\lb \xi_1 \rb^{-s}\lb \xi_2 \rb^{-k-\frac{3}{2}+5\theta} & \text{if} \quad |\xi_2| \gg |\gamma|^{\frac12}|\delta|^{-\frac12}\lambda, \\
        \lb \xi \rb^{s}\lb \xi_1 \rb^{-s}\lb \xi_2 \rb^{-k} & \text{if} \quad |\xi_2| \lesssim |\gamma|^{\frac12}|\delta|^{-\frac12}\lambda.
     \end{cases}
\end{align*}
Recall that we are assuming $0 \leq s < k+\frac{3}{2}$. For $|\xi_2| \gg |\gamma|^{\frac12}|\delta|^{-\frac12}\lambda$ and $\theta>0$ such that $5\theta < k-s+\frac{3}{2}$, we have
\begin{align} \label{eq: multiplier bound A}
     \frac{\chi_{P_2}(\xi_2,\tau_2)\lb \xi \rb^{s}}{\lb \xi_1 \rb^{s}\lb \xi_2 \rb^{k}\lb \tau_2 - p_{\lambda}(\xi_2) \rb^{\frac{3}{10}-\theta}} \lesssim 1.
\end{align}
For $|\xi_2| \lesssim |\gamma|^{\frac12}|\delta|^{-\frac12}\lambda$, we have
\begin{align} \label{eq: multiplier bound B}
    \frac{\chi_{P_2}(\xi_2,\tau_2)\lb \xi \rb^{s}}{\lb \xi_1 \rb^{s}\lb \xi_2 \rb^{k}\lb \tau_2 - p_{\lambda}(\xi_2) \rb^{\frac{3}{10}-\theta}} \lesssim \lb \xi_2 \rb^{s-k} \lesssim 1.
\end{align}
By (\ref{eq: Schrodinger L4}) and a simple change of variables, we can see that 
\begin{align} \label{eq: Schrodinger L4 dilated}
    \left \| \left [ \frac{\hat{f}(\xi,\tau)}{\lb \tau+\lambda^3 \xi^2\rb^{\frac38+}}\right ]^{\vee}\right \|_{L^{4}_{x,t}} \lesssim \lambda^{-\frac38}\|f\|_{L^2_{x,t}}.
\end{align}
Using (\ref{eq: multiplier bound A}), (\ref{eq: multiplier bound B}), (\ref{eq: Schrodinger L4 dilated}) and the $L^4L^4L^2$-H\"older inequality, we obtain the desired bound.
\end{proof}

\subsection{Local well-posedness}

\begin{proof}[Proof of Theorem \ref{thm: LWP}]
We only consider the case  $0 \leq s < k+\frac{3}{2}$, $k \geq -2$, and $k<\min(8s+1,s+2)$. The higher regularity case can be treated without using the dilation argument. We follow the argument given in \cite{GW}. For $0<\lambda<1$, we set $(u_{0,\lambda}(x),v_{0,\lambda}(x)):=(\lambda^4 u_{0}(\lambda x),\lambda^4 v_{0}(\lambda x))$ and $(u_{\lambda}(x,t), v_{\lambda}(x,t)):=(\lambda^4 u(\lambda x, \lambda^5 t),\lambda^4 v(\lambda x, \lambda^5 t)).$
Then $(u_{\lambda},v_{\lambda})$ solves
\begin{equation}\label{eq: SK dilated}
\begin{cases}
i\partial_{t} u_{\lambda} + \lambda^3 \partial^{2}_{x}u_{\lambda} = \alpha \lambda u_{\lambda}v_{\lambda} +\beta \lambda^{-3 }|u_{\lambda}|^2u_{\lambda},  \\
\partial_{t} v_{\lambda} + \gamma \lambda^{2} \partial^{3}_{x}v_{\lambda}-\delta \partial^{5}_{x}v_{\lambda} + v_{\lambda}\partial_{x}v_{\lambda} = \varepsilon\partial_{x}|u_{\lambda}|^2, \\
u_{\lambda}(x,0) = u_{0,\lambda}(x), \quad v_{\lambda}(x,0) = v_{0,\lambda}(x).
\end{cases}
\end{equation}
Let
\begin{align} \label{eq: LWP operators}
    \begin{cases}
       \Gamma_1(u,v):=\psi_T S_{\lambda}(t)u_{0,\lambda}-i\psi_T\int_{0}^{t}S_{\lambda}(t-t')(\alpha \lambda uv +\beta \lambda^{-3} |u|^2u) dt', \\
       \Gamma_2(u,v):=\psi_T W_{\lambda}(t)v_{0,\lambda}+\psi_T\int_{0}^{t}W_{\lambda}(t-t')(-v\partial_{x}v+\varepsilon\partial_{x}|u|^2) dt'.
    \end{cases}
\end{align}
Consider the ball
\begin{align*}
    B:= \{(u,v) \in X^{s,\frac12 +\theta}_{\lambda} \times Y^{k}_{\lambda,\theta} : \|u\|_{X^{s,\frac12 +\theta}_{\lambda}} \leq M \quad \text{and} \quad \|v\|_{Y^{k}_{\lambda,\theta}} \leq \epsilon_0 \}
\end{align*}
where $\epsilon_{0}$ and $M$ will be chosen later. For $(u,v) \in B$ we have
\begin{align*}
     \|\Gamma_1(u,v)\|_{X^{s,\frac12+ \theta}_{\lambda}} 
     &\leq C_1 \|u_{0,\lambda}\|_{H^{s}}+C_2 T^{\theta}\left ( \lambda \| uv \|_{X^{s,-\frac12 +2\theta}_{\lambda}} + \lambda^{-3} \| |u|^2 u \|_{X^{s,-\frac12 +2\theta}_{\lambda}} \right ) \\
     & \leq C_1 \|u_{0,\lambda}\|_{H^{s}}+C_3 T^{\theta}\left ( \lambda^{1-\kappa} \|u\|_{X^{s,\frac12+\theta}_{\lambda}}\|v\|_{Y^{k}_{\lambda,\theta}}+ \lambda^{-3-\kappa} \| u \|_{X^{s,\frac12 +\theta}_{\lambda}}^3 \right ) \\
     & \leq C_1 \|u_{0,\lambda}\|_{H^{s}}+C_3 T^{\theta}\left ( \lambda^{1-\kappa} M \epsilon_0 + \lambda^{-3-\kappa} M^3 \right ),
\end{align*}
and by (\ref{eq: Y embedding}),
\begin{align*}
     \|\Gamma_2(u,v)\|_{Y^{k}_{\lambda,\theta}} 
     & \leq \! \begin{multlined}[t]
          C_1 \|v_{0,\lambda}\|_{H^{k}} \\ +\left \|\psi_T\int_{0}^{t}W_{\lambda}(t-t') v\partial_{x}v dt' \right \|_{Y^{k}_{\lambda,\theta}} + \left \|\psi_T\int_{0}^{t}W_{\lambda}(t-t')\partial_{x}|u|^2 dt' \right \|_{Y^{k,\frac12 +2\theta}_{\lambda}} 
     \end{multlined}\\
     &\leq C_1 \|v_{0,\lambda}\|_{H^{k}}+C_2\| v \partial_x v \|_{Z^{k}_{\lambda,\theta}} + C_2 T^{\theta} \| \partial_x |u|^2 \|_{Y^{k,-\frac12 +\theta}_{\lambda}} \\
     & \leq C_1 \|v_{0,\lambda}\|_{H^{k}}+C_3 \|v\|_{Y^{k}_{\lambda,\theta}}^2+C_3 T^{\theta} \lambda^{-\kappa}\| u \|_{X^{s,\frac12 +\theta}_{\lambda}}^2 \\
     & \leq C_1 \|v_{0,\lambda}\|_{H^{k}}+C_3 \epsilon_{0}^2+C_3 T^{\theta} \lambda^{-\kappa}M^2.
\end{align*}
We can similarly estimate the difference $\|\Gamma_{1}(u,v)-\Gamma_{1}(\t u,\t v) \|_{X^{s,\frac12+ \theta}_{\lambda}}+\|\Gamma_{2}(u,v)-\Gamma_{2}(\t u,\t v)\|_{Y^{k}_{\lambda,\theta}}$. Let $M=2C_1 \|u_{0,1}\|_{H^{s}}$. Then for all $0<\lambda<1$ we have $C_{1}\|u_{0,\lambda}\|_{H^{s}} \leq M/2$. Take $\epsilon_0 >0$ such that $4C_3\epsilon_0 < 1$. Note that 
\begin{align*}
    \|v_{0,\lambda}\|_{H^k} \leq \lambda^{\frac32}\|v_0\|_{H^k}
\end{align*}
for $k \geq -2$. Take sufficiently small $\lambda$ such that $C_1 \|v_{0,\lambda}\|_{H^{k}} < \epsilon_0/2$. Then for sufficiently small $T$, we can see that the map $(\Gamma_1,\Gamma_2)$ is a contraction in $B$. Undoing the scaling, we obtain a local solution to the system (\ref{eq: SK}). For the uniqueness of the solutions and the Lipschitz continuity of the solution map, see \cite{Kato}.
\end{proof}

\section{Global theory} \label{sec: GWP}

\subsection{Linear and bilinear estimates}

\begin{lem}\cite[Proposition 2.1]{GP} \label{lem: group, duhamel}
Suppose $0<T\leq1$. If $0 \leq b,c \leq \frac12$ and $b+c \leq 1$, then 
\begin{enumerate}[label=(\roman*)]
    \item $\| \psi_T S(t)u_0 \|_{X^{0,b,1}} \lesssim T^{\frac12 -b}\|u_0\|_{L^2}$.
    \item $\| \psi_T W(t)v_0 \|_{Y^{0,b,1}} \lesssim T^{\frac12 -b}\|v_0\|_{L^2}$.
    \item $\| \psi_T \int_{0}^{t}S(t-t')F dt' \|_{X^{0,b,1}} \lesssim T^{1-b-c}\|F\|_{X^{0,-c,\infty}}$.
    \item $\| \psi_T \int_{0}^{t}W(t-t')F dt' \|_{Y^{0,b,1}} \lesssim T^{1-b-c}\|F\|_{Y^{0,-c,\infty}}$.
    \end{enumerate}
\end{lem}

\begin{lem} \label{lem: multilinear}
We have
    \begin{enumerate}[label=(\roman*)]
    \item $\| uv \|_{X^{0,-\frac14,\infty}} \lesssim \|u\|_{X^{0,\frac38,1}}\|v\|_{Y^{0,\frac38,1}}$.
    \item $\| u_1 \overline{u_2} u_3 \|_{X^{0,-\frac14,\infty}} \lesssim \|u_1\|_{X^{0,\frac38,1}}\|u_2\|_{X^{0,\frac38,1}}\|u_3\|_{X^{0,\frac38,1}}$.
    \item $\| \partial_{x}(v_1 v_2) \|_{Y^{0,-\frac14,\infty}} \lesssim \|v_1\|_{Y^{0,\frac38,1}}\|v_2\|_{Y^{0,\frac38,1}}$.
    \item $\| \partial_{x}(u_1 \overline{u_2}) \|_{Y^{0,-\frac14,\infty}} \lesssim \|u_1\|_{X^{0,\frac38,1}}\|u_2\|_{X^{0,\frac38,1}}$.
    \end{enumerate}
\end{lem}

\begin{proof} [Proof of (i)]
By duality, it suffices to bound the quantity
\begin{align} \label{eq: uv dual}
\iint_{\begin{subarray}{l}\\ \xi=\xi_{1}+\xi_{2}\\ \tau=\tau_{1}+\tau_{2} \end
{subarray}} \frac{|\hat{f}(\xi, \tau)|}{\lb \tau + \xi^2 \rb^{\frac14}}\frac{|\hat{g}(\xi_1,\tau_1)|}{\lb \tau_1 + \xi_1^2 \rb^{\frac38}}\frac{|\hat{h}(\xi_2,\tau_2)|}{\lb \tau_2 - p(\xi_2) \rb^{\frac38}}    
\end{align}
by $\| f\|_{X^{0,0,1}} \|g\|_{X^{0,0,1}} \|h\|_{Y^{0,0,1}}$. 

\noindent \textit{Case 1: $|\xi_2| \lesssim 1$.}
Applying Cauchy-Schwarz inequality, we can bound (\ref{eq: uv dual}) by
\begin{align*}
        & \int_{\R^4} \frac{|\hat{f}(\xi, \tau)|}{\lb \tau + \xi^2 \rb^{\frac14}}\frac{|\hat{g}(\xi-\xi_2, \tau-\tau_2)|}{\lb (\tau-\tau_2) + (\xi-\xi_2)^2 \rb^{\frac38}}\frac{ |\hat{h}(\xi_{2}, \tau_{2})|}{\lb \tau_2 - p(\xi_2) \rb^{\frac38}} d\xi d\xi_2 d\tau d\tau_2 \\
        & \leq \begin{multlined}[t] \left (\int_{\R^4} |\hat{g}(\xi-\xi_2,\tau-\tau_2)|^2 |\hat{h}(\xi_{2}, \tau_{2})|^2 d\xi d\xi_2 d\tau d\tau_2\right)^{\frac12} \\
        \times \left (\int_{\R^4} \frac{|\hat{f}(\xi, \tau)|^2}{\lb \tau + \xi^2 \rb^{\frac12} \lb (\tau-\tau_2) + (\xi-\xi_2)^2 \rb^{\frac34} \lb \tau_2 - p(\xi_2) \rb^{\frac34}} d\xi d\xi_2 d\tau d\tau_2\right )^{\frac12} \end{multlined} \\
        & \lesssim \| g\|_{L^2} \|h\|_{L^2} \left (\int_{\R^3} \frac{|\hat{f}(\xi, \tau)|^2}{\lb \tau - \xi^2 \rb^{\frac12} \lb \tau - (\xi-\xi_2)^2- p(\xi_2) \rb^{\frac12}} d\xi d\xi_2 d\tau \right )^{\frac12} \\
        & \lesssim \| g\|_{L^2} \|h\|_{L^2} \left (\int_{\R^3} \frac{|\hat{f}(\xi, \tau)|^2}{\lb \xi_2 (\delta \xi_2^4 +\gamma \xi_{2}^2-\xi_2+2\xi)\rb^{\frac12}} d\xi d\xi_2 d\tau \right )^{\frac12} \\
        & \leq \left ( \sup_{\xi} \int \frac{ d\xi d\xi_2 d\tau}{\lb  \xi_2 (\delta \xi_2^4 +\gamma \xi_{2}^2-\xi_2+2\xi)\rb^{\frac12}} \right )^{\frac12}\| f\|_{L^2} \|g\|_{L^2} \|h\|_{L^2}.
\end{align*}
We used Lemma \ref{CalcLem} in $\tau_2$ integration in the second inequality, and 
\begin{multline*}
    \lb \tau - \xi^2 \rb^{\frac12}\lb \tau - (\xi-\xi_2)^2- p(\xi_2) \rb^{\frac12} \\ \geq \lb \tau - \xi^2 \rb^{\frac12}\lb \tau - (\xi-\xi_2)^2- p(\xi_2) \rb^{\frac12}  \gtrsim \lb (\tau - \xi^2)-(\tau - \xi_1^2- p(\xi-\xi_1)) \rb^{\frac12} \\
    = \lb \xi_2 (\delta \xi_2^4 +\gamma \xi_{2}^2-\xi_2+2\xi)\rb^{\frac12}
\end{multline*}
in the third inequality. Therefore, it suffices to show that the quantity
\begin{align} \label{eq: uv reduced}
    \sup_{\xi} \int_{|\xi_2| \lesssim 1}\frac{ d\xi_2 }{\lb \xi_2 (\delta \xi_2^4 +\gamma \xi_{2}^2-\xi_2+2\xi)\rb^{\frac12}}
\end{align}
is finite. Since $\lb \xi_2 (\delta \xi_2^4 +\gamma \xi_{2}^2-\xi_2+2\xi)\rb \gtrsim |\xi_2|\lb \delta \xi_2^4 +\gamma \xi_{2}^2-\xi_2+2\xi \rb$ we have
\begin{align*}
    (\ref{eq: uv reduced}) \lesssim \sup_{\xi} \int_{|\xi_2| \lesssim 1}\frac{ d\xi_2 }{|\xi_2|^{\frac12}\lb \delta \xi_2^4 +\gamma \xi_{2}^2-\xi_2+2\xi\rb^{\frac12}}<\infty.
\end{align*}

\noindent \textit{Case 2: $|\xi_2| \gg 1$.}

In this case, the desired bound is immediate from (\ref{eq: Schrodinger L3}) and (\ref{eq: Kawahara L3}) since
\begin{align*}
    (\ref{eq: uv dual}) 
    &\lesssim \left \| \left [ \frac{|\hat{f}(\xi,\tau)|}{\lb \tau + \xi^{2} \rb^{\frac14}}\right ]^{\vee} \right \|_{L^{3}_{t,x}} \left \| \left [ \frac{\hat{g}(\xi,\tau)}{\lb \tau + \xi^{2} \rb^{\frac38}}\right ]^{\vee} \right \|_{L^{3}_{t,x}} \left \| \left [ \frac{\chi_{\{|\xi|\geq N\}}\hat{h}(\xi,\tau)}{\lb \tau -p(\xi) \rb^{\frac38}}\right ]^{\vee} \right \|_{L^{3}_{t,x}}\\
    &\lesssim \| f\|_{X^{0,0,1}} \|g\|_{X^{0,0,1}} \|h\|_{Y^{0,0,1}}.
\end{align*}
\end{proof}

\begin{proof} [Proof of (ii)]
By duality, it is enough to bound the quantity
\begin{align}  \label{eq: |u|^2u dual}
         \iint_{\begin{subarray}{l}\\ \xi=\xi_{1}+\xi_{2}+\xi_{3}\\ \tau=\tau_{1}+\tau_{2}+\tau_{2} \end{subarray}} \frac{|\hat{f}(\xi,\tau)|}{\lb \tau+\xi^2 \rb^{\frac14}}\frac{|\hat{g}(\xi_1,\tau_1)|}{\lb \tau_1+\xi_1^2 \rb^{\frac38}}\frac{|\hat{h}(-\xi_{2},-\tau_{2})|}{\lb \tau_2-\xi_2^2\rb^{\frac38}}\frac{|\hat{k}(\xi_{3},\tau_{3})|}{\lb \tau_3+\xi_3^2\rb^{\frac38}}
\end{align}
by $\| f \|_{X^{0,0,1}}\| g \|_{X^{0,0,1}}\| h \|_{X^{0,0,1}}\| k \|_{X^{0,0,1}}$. This follows directly from (\ref{eq: Schrodinger interpolation for cubic 1}), (\ref{eq: Schrodinger interpolation for cubic 2}), and the H\"older inequality.
\end{proof}

\begin{proof} [Proof of (iii)]
By duality, it is enough to bound the quantity
\begin{align}  \label{eq: vv_x dual}
         \iint_{\begin{subarray}{l}\\ \xi=\xi_{1}+\xi_{2}\\ \tau=\tau_{1}+\tau_{2} \end{subarray}} \frac{|\xi||\hat{f}(\xi,\tau)|}{\lb \tau- p(\xi) \rb^{\frac14}}\frac{  |\hat{g}(\xi_1,\tau_1)|}{\lb \tau_{1}-p(\xi_1) \rb^{\frac38}}\frac{ |\hat{h}(\xi_{2},\tau_{2})|}{\lb \tau_{2}- p(\xi_{2})\rb^{\frac38}}
\end{align}
by $\| f \|_{Y^{0,0,1}}\| g \|_{Y^{0,0,1}}\| h \|_{Y^{0,0,1}}$.
By the Cauchy-Schwarz inequality, it suffices to demonstrate that the quantity
\begin{align} \label{eq: vv_x reduced 1}
        \sup_{\xi_{1}} \int \frac{ \lb \xi \rb^{2} d\xi}{\lb 5\delta \xi\xi_1(\xi-\xi_1)(\xi^2-\xi\xi_1+\xi_1^2)+3\gamma\xi\xi_{1}(\xi-\xi_1)\rb^{\frac12}} 
\end{align}
is finite. Alternatively, it is enough to show that
\begin{align} \label{eq: vv_x reduced 2}
        \sup_{\xi} \int \frac{ \lb \xi \rb^{2}     d\xi_1}{\lb 5\delta \xi\xi_1(\xi-\xi_1)(\xi^2-\xi\xi_1+\xi_1^2)+3\gamma\xi\xi_{1}(\xi-\xi_1)\rb^{\frac12}} 
\end{align}
is finite. 

If $|\xi| \lesssim 1$ and $|\xi_1| \lesssim 1$, then there is nothing to prove.

\noindent \textit{Case 1: $|\xi| \lesssim 1$ and $|\xi_1| \gg 1$.}
\begin{align*}
    (\ref{eq: vv_x reduced 2}) \lesssim \sup_{\xi_1} \int_{|\xi| \lesssim 1} \frac{d\xi}{|\xi|^{\frac12}\lb \xi_1\rb^{2}} < \infty.
\end{align*}

\noindent \textit{Case 2: $|\xi_1| \lesssim 1$ and $|\xi| \gg 1$.}
\begin{align*}
    (\ref{eq: vv_x reduced 1}) \lesssim  \sup_{\xi} \int \frac{\lb \xi \rb^{2} d\xi_1}{|\xi_1|^{\frac12}\lb\xi\rb^{2}}  < \infty.
\end{align*}

\noindent \textit{Case 3: $|\xi-\xi_1| \lesssim 1$, $|\xi| \gg 1$ and $|\xi_1| \gg 1$.}
 \begin{align*}
 (\ref{eq: vv_x reduced 1}) \lesssim  \sup_{\xi} \int \frac{\lb \xi \rb^{2} d\xi_1}{|\xi-\xi_1|^{\frac12}\lb\xi\rb^{2}}  < \infty.
 \end{align*}
 
 In the rest of this proof, we always assume that $|\xi-\xi_1| \gg 1$, $|\xi_1| \gg 1$ and $|\xi-\xi_1| \gg 1$. Let $M:=\max(|\tau-p(\xi)|,|\tau_1-p(\xi_1)|,|\tau_2-p(\xi_2)|)$. Then $M\gtrsim|5\delta \xi\xi_1(\xi-\xi_1)(\xi^2-\xi\xi_1+\xi_1^2)+3\gamma\xi\xi_{1}(\xi-\xi_1)|$. Let
\begin{align*}
    P(\xi,\tau,\xi_1,\tau_1,\xi_2,\tau_2):=\frac{\lb \xi \rb|\hat{f}(\xi,\tau)|}{\lb \tau- p(\xi) \rb^{\frac14}}\frac{|\hat{g}(\xi_1,\tau_1)|}{\lb \tau_{1}-p(\xi_1) \rb^{\frac38}}\frac{|\hat{h}(\xi_{2},\tau_{2})|}{\lb \tau_{2}- p(\xi_2)\rb^{\frac38}}.
\end{align*}

\noindent \textit{Case 4: $|\xi| \gg |\xi_1|$.}

In this case, we have $M \gtrsim |\xi|^4|\xi_1|$ and $|\xi| \approx |\xi_2|$.

(i) $M=|\tau-p(\xi)|$. We have
\begin{align*}
         P(\xi,\tau,\xi_1,\tau_1,\xi_2,\tau_2) 
         \lesssim |\hat{f}(\xi,\tau)|\frac{|\hat{g}(\xi_1,\tau_1)|}{\lb \tau_{1}-p(\xi_1) \rb^{\frac38}}\frac{|\hat{h}(\xi_{2},\tau_{2})|}{\lb \tau_{2}- p(\xi_2)\rb^{\frac38}}.
\end{align*}
Using (\ref{eq: Kawahara L4}), we can apply the $L^2L^4L^4$-H\"older inequality.

(ii) $M=|\tau_1-p(\xi_1)|$. We have,
\begin{align*}
         P(\xi,\tau,\xi_1,\tau_1,\xi_2,\tau_2) 
         &\lesssim \frac{\lb \xi \rb^{\frac12}|\hat{f}(\xi,\tau)|}{\lb \tau-p(\xi) \rb^{\frac14}}\frac{|\hat{g}(\xi_1,\tau_1)|}{\lb \tau_{1}-p(\xi_1) \rb^{\frac14}}\frac{|\hat{h}(\xi_{2},\tau_{2})|}{\lb \tau_{2}- p(\xi_2)\rb^{\frac38}} \\
         &\approx \frac{\lb \xi \rb^{\frac14}|\hat{f}(\xi,\tau)|}{\lb \tau-p(\xi) \rb^{\frac14}}\frac{|\hat{g}(\xi_1,\tau_1)|}{\lb \tau_{1}-p(\xi_1) \rb^{\frac14}}\frac{\lb \xi_2 \rb^{\frac14}|\hat{h}(\xi_{2},\tau_{2})|}{\lb \tau_{2}- p(\xi_2)\rb^{\frac38}}.
\end{align*}
Using (\ref{eq: Kawahara L3}), we can apply the $L^3L^3L^3$-H\"older inequality.

(iii) $M=|\tau_2-p(\xi_2)|$. Similar to (ii).

\noindent \textit{Case 5: $|\xi_1| \gg |\xi|$.}

In this case, we have $M \gtrsim |\xi||\xi_1|^4$. Also note that $|\xi_1| \approx |\xi_2|$.

(i) $M=|\tau-p(\xi)|$. We have
\begin{align*}
     P(\xi,\tau,\xi_1,\tau_1,\xi_2,\tau_2) 
    &\lesssim \lb \xi \rb^{\frac34}|\hat{f}(\xi,\tau)|\frac{\lb \xi_1 \rb^{-1}|\hat{g}(\xi_1,\tau_1)|}{\lb \tau_{1}-p(\xi_1) \rb^{\frac38}}\frac{|\hat{h}(\xi_{2},\tau_{2})|}{\lb \tau_{2}- p(\xi_2)\rb^{\frac38}} \\
    &\lesssim |\hat{f}(\xi,\tau)|\frac{\lb \xi_1 \rb^{-\frac14}|\hat{g}(\xi_1,\tau_1)|}{\lb \tau_{1}-p(\xi_1) \rb^{\frac38}}\frac{|\hat{h}(\xi_{2},\tau_{2})|}{\lb \tau_{2}- p(\xi_2)\rb^{\frac38}} \\
    &\lesssim |\hat{f}(\xi,\tau)|\frac{|\hat{g}(\xi_1,\tau_1)|}{\lb \tau_{1}-p(\xi_1) \rb^{\frac38}}\frac{|\hat{h}(\xi_{2},\tau_{2})|}{\lb \tau_{2}- p(\xi_2)\rb^{\frac38}}.
\end{align*}
Using (\ref{eq: Kawahara L4}), we can apply the $L^2L^4L^4$-H\"older inequality.

(ii) $M=|\tau_1-p(\xi_1)|$. We have
\begin{align*}
     P(\xi,\tau,\xi_1,\tau_1,\xi_2,\tau_2) 
    &\lesssim \frac{\lb \xi \rb^{\frac78}|\hat{f}(\xi,\tau)|}{\lb \tau-p(\xi) \rb^{\frac14}}\frac{\lb \xi_1 \rb^{-\frac12}|\hat{g}(\xi_1,\tau_1)|}{\lb \tau_{1}-p(\xi_1) \rb^{\frac14}}\frac{ |\hat{h}(\xi_{2},\tau_{2})|}{\lb \tau_{2}- p(\xi_2)\rb^{\frac38}} \\
    &\lesssim \frac{\lb \xi \rb^{\frac14}|\hat{f}(\xi,\tau)|}{\lb \tau-p(\xi) \rb^{\frac14}}\frac{\lb \xi_1 \rb^{\frac18}|\hat{g}(\xi_1,\tau_1)|}{\lb \tau_{1}-p(\xi_1) \rb^{\frac14}}\frac{|\hat{h}(\xi_{2},\tau_{2})|}{\lb \tau_{2}- p(\xi_2)\rb^{\frac38}}.
\end{align*}
Using (\ref{eq: Kawahara L3}), we can apply the $L^3L^3L^3$-H\"older inequality.

(iii) $M=|\tau_2-p(\xi_2)|$. Similar to (ii).

\noindent \textit{Case 6: $|\xi_1| \approx |\xi|$.}

In this case, we have $M \gtrsim |\xi|^4$. 

(i) $M=|\tau-p(\xi)|$. We have
\begin{align*}
     P(\xi,\tau,\xi_1,\tau_1,\xi_2,\tau_2) 
     \lesssim |\hat{f}(\xi,\tau)|\frac{|\hat{g}(\xi_1,\tau_1)|}{\lb \tau_{1}-p(\xi_{1}) \rb^{\frac38}}\frac{ |\hat{h}(\xi_{2},\tau_{2})|}{\lb \tau_{2}- p(\xi_{2})\rb^{\frac38}} 
\end{align*}
Using (\ref{eq: Kawahara L4}), we can apply the $L^2L^4L^4$-H\"older inequality.

(ii) $M=|\tau_1-p(\xi_1)|$. We have
\begin{align*}
     P(\xi,\tau,\xi_1,\tau_1,\xi_2,\tau_2) 
     &\approx \frac{\lb \xi \rb^{\frac12}|\hat{f}(\xi,\tau)|}{\lb \tau- p(\xi) \rb^{\frac14}}\frac{  |\hat{g}(\xi_1,\tau_1)|}{\lb \tau_{1}-p(\xi_{1}) \rb^{\frac14}}\frac{ |\hat{h}(\xi_{2},\tau_{2})|}{\lb \tau_{2}- p(\xi_{2})\rb^{\frac38}} \\
     &\approx \frac{\lb \xi \rb^{\frac14}|\hat{f}(\xi,\tau)|}{\lb \tau- p(\xi) \rb^{\frac14}}\frac{ \lb \xi \rb^{\frac14} |\hat{g}(\xi_1,\tau_1)|}{\lb \tau_{1}-p(\xi_{1}) \rb^{\frac14}}\frac{ |\hat{h}(\xi_{2},\tau_{2})|}{\lb \tau_{2}- p(\xi_{2})\rb^{\frac38}}. 
\end{align*}
Using (\ref{eq: Kawahara L3}), we can apply the $L^3L^3L^3$-H\"older inequality.

(iii) $M=|\tau_2-p(\xi_2)|$. Similar to (ii).
\end{proof}

\begin{proof} [Proof of (iv)]
By duality, it is enough to bound the quantity
\begin{align}  \label{eq: |u|^2_x dual}
         \iint_{\begin{subarray}{l}\\ \xi=\xi_{1}+\xi_{2}\\ \tau=\tau_{1}+\tau_{2} \end{subarray}} \frac{|\xi||\hat{f}(\xi,\tau)|}{\lb \tau- p(\xi) \rb^{\frac14}}\frac{|\hat{g}(\xi_1,\tau_1)|}{\lb \tau_{1}+\xi_{1}^{2} \rb^{\frac38}}\frac{|\hat{h}(-\xi_{2},-\tau_{2})|}{\lb \tau_{2}-\xi_{2}^{2}\rb^{\frac38}}
\end{align}
by $\| f \|_{Y^{0,0,1}}\| g \|_{X^{0,0,1}}\| h \|_{X^{0,0,1}}$.

\noindent \textit{Case 1: $|\xi| \lesssim 1$.}

By the Cauchy-Schwartz inequality, it suffices to demonstrate that the quantity
\begin{align} \label{eq: |u|^2_x reduced}
        \sup_{\xi_{1}}\int_{|\xi| \lesssim 1}\frac{\lb \xi \rb^{2} d\xi}{\lb\xi(\delta \xi^{4}+\gamma \xi^{2}+\xi-2\xi_{1})\rb^{\frac12}} 
\end{align}
is finite. Since $|\xi| \lesssim 1$, we have $\lb\xi(\delta \xi^{4}+\gamma \xi^{2}+\xi-2\xi_{1})\rb \gtrsim \lb \xi \xi_1 \rb \gtrsim |\xi|\lb \xi_1 \rb$ . Thus,
\begin{align*}
    (\ref{eq: |u|^2_x reduced}) \lesssim \sup_{\xi_{1}} \frac{1}{\lb \xi_{1} \rb^{\frac12}}\int_{|\xi| \lesssim 1} \frac{1}{|\xi|^{\frac12}} d\xi < \infty.
\end{align*}

 In the rest of this proof, we always assume that $|\xi| \gg 1$. Let $M=\max(|\tau- p(\xi)|,|\tau_{1}+ \xi_{1}^{2}|,|\tau_{2}-\xi_{2}^{2}|)$. Then $M \gtrsim |\xi(\delta \xi^{4}+\gamma \xi^{2} +\xi -2\xi_1)|$. Let
\begin{align*}
    P(\xi,\tau,\xi_1,\tau_1,\xi_2,\tau_2):=\frac{\lb \xi \rb|\hat{f}(\xi,\tau)|}{\lb \tau- p(\xi) \rb^{\frac14}}\frac{|\hat{g}(\xi_1,\tau_1)|}{\lb \tau_{1}+\xi_{1}^{2} \rb^{\frac38}}\frac{|\hat{h}(-\xi_{2},-\tau_{2})|}{\lb \tau_{2}-\xi_{2}^{2}\rb^{\frac38}}.
\end{align*}

\noindent \textit{Case 2: $|\xi_1| \geq 2|\delta| |\xi|^{4}$.}

In this case we have $M \gtrsim |\xi||\xi_1|$. 

(i) $M=|\tau-p(\xi)|$. Since $M \gtrsim |\xi||\xi_1| \gtrsim |\xi|^5$, we have 
\begin{align*}
     P(\xi,\tau,\xi_1,\tau_1,\xi_2,\tau_2) \lesssim |\hat{f}(\xi,\tau)|\frac{|\hat{g}(\xi_1,\tau_1)|}{\lb \tau_{1}+\xi_{1}^{2} \rb^{\frac38}}\frac{|\hat{h}(-\xi_{2},-\tau_{2})|}{\lb \tau_{2}-\xi_{2}^{2}\rb^{\frac38}}.
\end{align*}
Using (\ref{eq: Schrodinger L4}), we can apply the $L^2L^4L^4$-H\"older inequality.

(ii) $M=|\tau_1+\xi_1^2|$. We have 
\begin{align*}
     P(\xi,\tau,\xi_1,\tau_1,\xi_2,\tau_2) 
     &\lesssim \frac{\lb \xi \rb^{\frac34}|\hat{f}(\xi,\tau)|}{\lb \tau- p(\xi) \rb^{\frac14}}\frac{\lb \xi_1 \rb^{-\frac14}|\hat{g}(\xi_1,\tau_1)|}{\lb \tau_{1}+\xi_{1}^{2} \rb^{\frac18}}\frac{|\hat{h}(-\xi_{2},-\tau_{2})|}{\lb \tau_{2}-\xi_{2}^{2}\rb^{\frac38}} \\
     &\lesssim \frac{\lb \xi_1 \rb^{\frac{3}{16}}|\hat{f}(\xi,\tau)|}{\lb \tau- p(\xi) \rb^{\frac14}}\frac{\lb \xi_1 \rb^{-\frac14}|\hat{g}(\xi_1,\tau_1)|}{\lb \tau_{1}+\xi_{1}^{2} \rb^{\frac18}}\frac{|\hat{h}(-\xi_{2},-\tau_{2})|}{\lb \tau_{2}-\xi_{2}^{2}\rb^{\frac38}} \\
     &\lesssim \frac{|\hat{f}(\xi,\tau)|}{\lb \tau- p(\xi) \rb^{\frac14}}\frac{|\hat{g}(\xi_1,\tau_1)|}{\lb \xi_1 \rb^{\frac{1}{16}}\lb \tau_{1}+\xi_{1}^{2} \rb^{\frac18}}\frac{|\hat{h}(-\xi_{2},-\tau_{2})|}{\lb \tau_{2}-\xi_{2}^{2}\rb^{\frac38}}.
\end{align*}
Then we can apply Lemma \ref{lem: key estimate 2}.

(iii) $M=|\tau_2-\xi_2^2|$. Similar to (ii).

\noindent \textit{Case 3: $|\xi_1| \leq 2|\delta| |\xi|^{4}$.}

In this case, we have $\lb \xi \rb \lesssim | 5\delta \xi^{4}+3\gamma \xi^{2}+2 \xi_1 |^{\frac{1}{4}}$. Therefore we can apply Lemma \ref{lem: key estimate}.

\end{proof}

\subsection{Global well-posedness}

Before we prove Theorem \ref{thm: GWP}, we need the following lemma which is a consequence of local well-posedness results \cite{CLMW, Kato} and the standard $L^2$-conservation law for the Kawahara equation.

\begin{lem} \label{lem: Kawahara conservation}
For any $z_0 \in L^{2}(\R)$, the equation 
\begin{equation} \label{eq: Kawahara}
\begin{cases}
\partial_{t} z + \gamma \partial^{3}_{x}z-\delta \partial^{5}_{x}z +z\partial_{x}z= 0, \\
z(x,0) = z_0(x).
\end{cases}
\end{equation}
is globally well-posed. Moreover, for any $t \in \R$, we have $\| z(t) \|_{L^{2}} = \|z_0 \|_{L^{2}}$.
\end{lem}

\begin{proof}[Proof of Theorem \ref{thm: GWP}]
Suppose that $(u,v)$ solves (\ref{eq: SK}) and $z$ solves (\ref{eq: Kawahara}) with $z_0=v_0$. Letting $w=v-z$, we see that $(u,v,w,z)$ solves the system
\begin{equation} \label{eq: difference equation}
\begin{cases}
i\partial_{t}u + \partial^{2}_{x}u = \alpha uv +\beta |u|^2u, \\
\partial_{t}v + \gamma \partial^{3}_{x}v -\delta \partial^{5}_{x}v + v\partial_{x}v = \varepsilon\partial_{x}|u|^2, \\
\partial_{t}w+ \gamma \partial^{3}_{x}w-\delta \partial^{5}_{x}w+\frac12 \partial_{x}((v+z)w)=\varepsilon\partial_{x}|u|^2, \\
\partial_{t}z + \gamma \partial^{3}_{x}z-\delta \partial^{5}_{x}z +z\partial_{x}z= 0, \\
u(x,0) = u_0(x), \quad v(x,0) =z(x,0)= v_0(x), \quad w(x,0)=0.
\end{cases}
\end{equation}
For $0<T\leq 1$ let
\begin{align} \label{eq: operators}
    \begin{cases}
       \Gamma_1(u,v,w,z)=\psi_T S(t)u_0-i\psi_T\int_{0}^{t}S(t-t')(\alpha uv +\beta |u|^2u) dt', \\
       \Gamma_2(u,v,w,z)=\psi_T W(t)v_0+\psi_T\int_{0}^{t}W(t-t')(-v\partial_{x}v+\varepsilon\partial_{x}|u|^2) dt', \\
       \Gamma_3(u,v,w,z)=\psi_T \int_{0}^{t}W(t-t')\left(-\frac12 \partial_{x}((v+z)w)+\varepsilon\partial_{x}|u|^2\right) dt', \\
       \Gamma_4(u,v,w,z)=\psi_T W(t)v_0+\psi_T\int_{0}^{t}W(t-t')(-z\partial_{x}z) dt'.
    \end{cases}
\end{align}
Then by Lemmas \ref{lem: group, duhamel} and \ref{lem: multilinear} we have
\begin{align*}
    &\|\Gamma_1(u,v,w,z)\|_{X^{0,\frac38,1}} 
     \leq C T^{\frac18}\|u_0\|_{L^2}+CT^{\frac38}\|u\|_{X^{0,\frac38,1}}\|v\|_{Y^{0,\frac38,1}}+CT^{\frac38}\|u\|_{X^{0,\frac38,1}}^3 \\
    &\|\Gamma_2(u,v,w,z)\|_{Y^{0,\frac38,1}} 
     \leq CT^{\frac18}\|v_0\|_{L^2}+CT^{\frac38}\|u\|_{X^{0,\frac38,1}}^2+CT^{\frac38}\|v\|_{Y^{0,\frac38,1}}^2, \\
    &\|\Gamma_3(u,v,w,z)\|_{Y^{0,\frac38,1}} 
     \leq \begin{multlined}[t] CT^{\frac38}\|v\|_{Y^{0,\frac38,1}}\|w\|_{Y^{0,\frac38,1}} \\
     +CT^{\frac38}\|z\|_{Y^{0,\frac38,1}}\|w\|_{Y^{0,\frac38,1}}+CT^{\frac38}\|u\|_{X^{0,\frac38,1}}^2,
     \end{multlined} \\
    & \|\Gamma_4(u,v,w,z)\|_{Y^{0,\frac38,1}} 
     \leq C T^{\frac18}\|v_0\|_{L^2}+CT^{\frac38}\|z\|_{Y^{0,\frac38,1}}^2.
\end{align*}
Let
\begin{multline*}
    B:=\{ u \in X^{0,\frac38,1} : \|u\|_{X^{0,\frac38,1}} \leq 2CT^{\frac18}\|u_0\|_{L^2}\} \times \{ v \in Y^{0,\frac38,1} : \|v\|_{Y^{0,\frac38,1}} \leq 2CT^{\frac18}\|v_0\|_{L^2}\} \\
    \times \{ w \in Y^{0,\frac38,1} : \|w\|_{Y^{0,\frac38,1}} \leq 8C^3T^{\frac58}\|u_0\|_{L^2}^2\} \times \{ z \in Y^{0,\frac38,1} : \|z\|_{Y^{0,\frac38,1}} \leq 2CT^{\frac18}\|v_0\|_{L^2}\}.
\end{multline*}
Since $\|u\|_{L^2}$ conserved, we may assume that $\|u_0\|_{L^2} \leq \|v_0\|_{L^2}$ and $C^3 T^{\frac58} \|u_0\|_{L^2}^3 \leq 1$. Take $0< T \leq 1$ such that $8C^{2}T^{\frac12}\|v_0\|_{L^2} \sim 1$. Then by a contraction argument, there exists a fixed point $(u,v,w,z)$ of (\ref{eq: operators}) in $B$. Using Lemma \ref{lem: group, duhamel} and Lemma \ref{lem: multilinear}, we have
\begin{align*}
        \|w\|_{C_t L^2_{x}([0,T] \times \R)} 
        &\lesssim \left \| \psi_T \int_{0}^{t}W(t-t')\left(-\frac12 \partial_{x}((v+z)w)+\varepsilon\partial_{x}|u|^2\right) dt'\right \|_{Y^{0,\frac12,1}} \\
        &\lesssim T^{\frac14} \left \| -\frac12 \partial_{x}((v+z)w)+\varepsilon\partial_{x}|u|^2 \right \|_{Y^{0,-\frac14,\infty}} \\
        &\lesssim T^{\frac14}\|v\|_{Y^{0,\frac38,1}}\|w\|_{Y^{0,\frac38,1}}+T^{\frac14}\|z\|_{Y^{0,\frac38,1}}\|w\|_{Y^{0,\frac38,1}}+T^{\frac14}\|u\|_{X^{0,\frac38,1}}^{2} \\
        & \lesssim T^{\frac12}\|u_0\|_{L^2}^2.
\end{align*}
Invoking $w=v-z$ and then using Lemma \ref{lem: Kawahara conservation}, we have
\begin{align*}
    \|v(T)\|_{L^2}\leq\|z(T)\|_{L^2}+\|w(T)\|_{L^2}\leq \|v_0\|_{L^2}+CT^{\frac12}\|u_0\|_{L^2}^2.
\end{align*}
Next, we update the initial data of the system (\ref{eq: difference equation}) to $u_0=u(T)$, $v_0=v(T)$, $z_0=v(T)$ and $w_0=0$, and then repeat the above estimate. We can iterate this process $m$ times on time intervals before the quantity $\|v\|_{L^2}$ doubles, where
\begin{align*}
    m \sim \frac{\|v_0\|_{L^2}}{T^{\frac12}\|u_0\|_{L^2}^2}.
\end{align*}
Since $mT \sim \|u_0\|_{L^2}^{-2}$, the solution can be globally extended.
\end{proof}

\section*{Acknowledgement}
The author would like to thank his advisor, Professor Nikolaos Tzirakis, for many helpful suggestions and comments.

\end{document}